\documentclass[11pt]{article}
\usepackage{amsmath,amsthm, amssymb, amsfonts, xspace, enumitem}
\usepackage{mathrsfs}
\usepackage[sort&compress,numbers,comma,square]{natbib}
\usepackage[active]{srcltx}
\usepackage{hyperref}

\newtheorem{theorem}{Theorem}
\newtheorem{prop}[theorem]{Proposition}
\newtheorem{lemma}[theorem]{Lemma}

\theoremstyle{remark}
\newtheorem*{remark}{Remark}

\def\Grp#1{\left(#1\right)}
\def\Cbr#1{\left\{#1\right\}}
\def\Sbr#1{\left[#1\right]}
\def\Abs#1{\left|#1\right|}

\def\Sp#1{\sp{(#1)}}

\def\th{\sp{\rm th}}

\def\uto{\uparrow}
\def\dto{\downarrow}
\def\Linf{\varliminf}

\def\nth#1{\frac{1}{#1}}
\def\lfrac#1#2{#1/#2}     

\def\Scr#1{\mathscr{#1}}

\def\cf#1{\mathbf{1}\!\Cbr{#1}}
\def\sumoi#1{\sum_{#1=1}^\infty}
\def\sumzi#1{\sum_{#1=0}^\infty}
\def\intii{\int_{-\infty}^\infty}
\def\intzi{\int_0^\infty}

\def\dd{\mathrm{d}}
\def\gv{\,|\,}
\def\iunit{\mathrm{i}}
\def\mean{\mathbb{E}}
\def\pr{\mathbb{P}}

\def\ralph{{1/\alpha}}

\def\LT#1{\widetilde #1}  
\def\FT#1{\widehat #1}  
\def\Re{\mathrm{Re}}
\def\res{\mathrm{Res}}

\def\toi{\to\infty}

\def\root{\varsigma}
\def\rx{\epsilon}

\def\Coms{\mathbb{C}}
\def\Nats{\mathbb{N}}
\def\Reals{\mathbb{R}}

\def\mittag{Mittag-Leffler\xspace}
\def\levy{L\'evy\xspace}

\makeatletter
\def\@cleandot{\@ifnextchar.{}{\@ifnextchar,{.}{\@ifnextchar;{.}{\@ifnextchar?{.}{\@ifnextchar:{.}{\@ifnextchar!{.}{\@ifnextchar'{.}{\@ifnextchar){.}{.\
                }}}}}}}}}
\def\pdf{p.d.f\@cleandot}
\def\lhs{l.h.s\@cleandot}
\def\rhs{r.h.s\@cleandot}
\makeatother
\begin{document}
\begin{center}
  \large
  \textbf{Law of two-sided exit by a
    spectrally positive strictly stable process\footnote{Research
      partially supported by NSF Grant DMS 1720218.}
  }  
  \\[1.5ex]
  \normalsize
  Zhiyi Chi \\
  Department of Statistics\\
  University of Connecticut \\
  Storrs, CT 06269, USA, \\[.5ex]
  E-mail: zhiyi.chi@uconn.edu \\[1ex]
  \today
\end{center}

\begin{abstract}
  For a spectrally positive strictly stable process with index in
  (1,\,2), the paper obtains i) the density of the time when the
  process makes first exit from an interval by hitting the interval's
  lower end point before jumping over its upper end point, and
  ii) the joint distribution of the time, the undershoot, and the jump
  of the process when it makes first exit the other way around.  For
  i), the density of the time of first exit is expressed as an
  infinite sum of functions, each the product of a polynomial
  and an exponential function, with all coefficients determined by
  the roots of a \mittag function.  For ii), conditional on the
  undershoot, the time and the jump of first exit are independent, and
  the marginal conditional densities of the time has similar features
  as i).

  \medbreak\noindent

  \medbreak\noindent

\end{abstract}

\section{Introduction}
The so-called exit problems, which concern the random event that a
stochastic process gets out of a set for the first time, occupy a
prominent place in the study of \levy processes.  For spectrally
positive \levy processes, years of intensive research have revealed
many remarkable facts about first exit from a bounded closed interval
\cite{bertoin:96:cup, kyprianou:14:sh, doney:07:sg-b}.  An essential
tool for many of the results is the scale function.  Since
the function can be analytically extended to $\Coms$ (\cite
{kyprianou:14:sh}, Lemma 8.3), it is amenable to treatments by
complex analysis.  By combining the scale function and residual
calculus, this paper obtains series expressions of the distribution of
first exit of a spectrally positive strictly stable process with index
in $(1,2)$.

Henceforth, without loss of generality, let $X$ be a \levy process
with 
\begin{align} \label{e:X}
  \mean [e^{-q X_t}] = e^{t q^\alpha}, \quad t>0,\ q>0,\ \alpha\in
  (1,2).
\end{align}
Given $b,c\in (0, \infty)$, first exit by $X$ from interval $[-b,c]$
consists of two possibilities: either the process makes a continuous
downward passage of $-b$ before it makes an upward jump across $c$ or
the other way around.  The probability of each possibility is
well-known (cf.\ \cite
{bertoin:96:cup}, Theorem VII.8).  Meanwhile, the scale functions of
the first exit have been known \cite{kyprianou:14:sh}.  On the other
hand, not much is known about the explicit joint probability density
function (\pdf) for the first exit.

In \cite{chi:18:tr}, by using Laplace transform, the distribution of
first upward passage of a fixed level by $X$ is obtained.  It turns
out that the method used there can be extended to first exit from an
interval.  Section \ref {s:lower} considers the time of first exit
from $[-b,c]$ at $-b$.  It will be shown that the probability density
function (\pdf) of the time has an expression in terms of the
residuals of a certain function at the roots of a \mittag function,
and as a result, is of the form $f(t) = \sum_\root p_\root(t)
e^{t\root}$, where the sum runs over the roots and for each root
$\root$, $p_\root(t)$ is a polynomial in $t$ whose coefficients are
determined by $\root$ and several \mittag functions.  For all but a
finite number of $\root$, $p_\root(t)$ is of order 0.  The result
provides a connection to some known results on first exit of a
standard Brownian motion.  It also highlights the importance of
gaining more information on the roots of \mittag functions
\cite{popov:13:jmathsci}.  Section \ref{s:upper} considers the joint
distribution of the time, the undershoot, and the jump of $X$ when its
first exit from $[-b,c]$ occurs at $c$.  It will be shown that
conditional on the undershoot, the time and the jump are independent.
This allows the joint distribution to be factorized into the marginal
\pdf of the undershoot, and the marginal conditional \pdf's of the
time and the jump, respectively.  The expression of the marginal
conditional \pdf of the time has similar features as the one of first
exit at the lower end.  This is in contrast to the power series
expression of the time of first upward passage of $c$ \cite
{chi:18:tr, bernyk:08:ap, simon:11:sto, peskir:08:ecp}, even though
the first passage can be regarded as the limit of first exit from
$[-b,c]$ as $b\toi$.  In both sections, the asymptotics of the time of
first exit as $t\to 0$ or $\infty$ are also considered.

The rest of the section fixes notation and collects some background
information for later use.

\paragraph{Integral transforms.} For $f\in L^1(\Reals)$, denote
its Laplace transform and Fourier transform, respectively, by
\begin{gather*}
  \LT f(z) = \int e^{-t z} f(t)\,\dd t
  \quad\text{and}\quad
  \FT f(\theta) = \LT f(-\iunit \theta), \quad \theta\in\Reals.
\end{gather*}
The domain of $\LT f$ is $\{z\in\Coms: e^{-t z} f(t)\in L^1(\dd t)\}$.
Similarly, for a finite measure $\mu$ on $\Reals$, denote its Laplace
transform and Fourier transform, respectively, by
\begin{gather*}
  \LT\mu(z) = \int e^{-t z} \mu(\dd t)
  \quad\text{and}\quad
  \FT\mu(\theta) = \LT\mu(-\iunit\theta),\quad \theta\in\Reals.
\end{gather*}
The domain of $\LT\mu$ is $\{z\in\Coms: e^{-t z} \in L^1(\dd t)\}$.

Let $z_0\in\Coms$.  Denote $U_r(z_0) = \{z\in\Coms: |z-z_0|<r\}$.  If
function $g$ is analytic in $U_r(z_0) \setminus\{z_0\}$ for some $r>0$
and has $z_0$ as a pole, possibly removable, then the residual of $g$
at $z_0$ is 
\begin{align*}
  \res(g(z), z_0) = c_1 = \frac1{2\pi\iunit} \oint_\gamma g(z)\,\dd z,
\end{align*}
where $\gamma$ is any counterclockwise simple closed contour in
$U_r(z_0)\setminus \{z_0\}$ (\cite{rudin:87:mcgraw}, p.~224).

\paragraph{Some properties of a \mittag function.}  A \mittag function
with parameters $a>0$ and $b\in\Coms$ is an entire function defined as
\begin{align*}
  E_{a, b}(z) = \sumzi  n \frac{z^n}{\Gamma(a n + b)},
  \quad z\in\Coms.
\end{align*}
Let $\Scr Z_{a, b} = \{z\in\Coms: E_{a,b}(z) = 0\}$.  The focus will
be mostly on $E_{\alpha, \alpha}(z)$ with $\alpha\in (1,2)$.  By
\cite{popov:13:jmathsci}, Theorems 1.2.1, 1.4.2, and 1.5.1, 
\begin{align} \label{e:asym-infty}
  E_{\alpha, \alpha}(z)
  =
  \alpha^{-1} z^{\ralph-1} \exp(z^\ralph)
  -
  \frac{(\alpha - 1)\alpha z^{-2}}{\Gamma(2-\alpha)} 
  + O(|z|^{-3}), \quad |z|\toi,
\end{align}
where the $O(\cdot)$ term is uniform in $\arg z$ and the principle
branch of $z$ is used so that $\arg(e^{\iunit\theta}) = \theta -
2k\pi$ for any $\theta\in \Reals$ with $k$ the unique integer
satisfying $2k-1 < \theta/\pi \le 2k+1$.  Furthermore, from
Theorems 2.1.1 and 4.2.1, and Chapter 6 in \cite{popov:13:jmathsci},
\begin{enumerate}[label={\arabic*)}, topsep=.5ex, parsep=0ex]
\item $\Scr Z_{\alpha,\alpha}$ has infinitely many elements and all
  those with large enough modulus are simple roots of $E_{\alpha,
    \alpha}(z)$ and can be enumerated as $\root_{\pm n}$, $n\ge N$,
  for some large $N$, such that
  \begin{align} \label{e:zero-asym}
    \root_{\pm n}
    =
    [1+o(1)](2\pi n)^\alpha e^{\pm\iunit\alpha\pi/2},
    \quad n\toi;
  \end{align}
\item  $\Scr Z_{\alpha,\alpha}\subset \{z: |\arg(z)| >
  \alpha\pi/2\}$; and
\item $E_{\alpha, \alpha}(z)$ has a finite positive number of real
  roots, all being negative.
\end{enumerate}

\paragraph{First passage time and hitting time.}
For $c>0$ and $x\in\Reals$, denote $T_c = \inf\{t>0: X_t>c\}$ and
$\tau_x = \inf\{t>0:\ X_t = x\}$.  Under the law of $X$,  $T_c <
\tau_c < \infty$ and $X_{T_c} > c > X_{T_c-}$ a.s.\ \cite
{simon:11:sto}, both $T_c$ and $\tau_x$ have \pdf's \cite
{bertoin:96:cup, bernyk:08:ap, peskir:08:ecp, simon:11:sto}, and given
$b>0$, as downward movement is continuous and 0 is regular for
$(-\infty, 0)$, $\tau_{-b} = \inf\{t: X_t < -b\}$ (\cite
{doney:07:sg-b}, Theorem 5.17), so the time of first exit from
$[-b,c]$ is $\min(\tau_{-b}, T_c)$.  When $\tau_{-b} < T_c$
(resp.\ $\tau_{-b} > T_c$), $X$ is said to make first exit from
$[-b,c]$ at the lower (resp.\ upper) end.  Denote by $g_t$ the \pdf of
$X_t$ and by $f_x$ the \pdf of $\tau_x$.  The distribution of $\tau_x$
is classical for $x<0$ (\cite {bertoin:96:cup}, Theorem VII.1) and is
also known for $x>0$ \cite  {simon:11:sto}.

We will rely on the scale function $W\Sp q$ of $-X$ for the derivation
(\cite{kyprianou:14:sh}, Chapter 8; also cf.\ \cite {bertoin:96:cup,
  doney:07:sg-b}).  For the spectrally negative strictly stable
process $-X$, 
\begin{align} \label{e:ss-scale}
  W\Sp q(x) = x^{\alpha-1}_+ E_{\alpha,\alpha}(q x^\alpha_+),  \quad
  q\ge 0,
\end{align}
where $x_+ = \max(x,0)$ (\cite{kyprianou:14:sh},  p.~250).

\section{Distribution of first exit time at lower end} \label{s:lower}

Given $c>0$ and $x<c$, denote by
\begin{align*}
  k_{x,c}(t)
  =
  \frac{\pr\{\tau_x\in\dd t,\, X_s<c\,\forall s\le t\}}{\dd t}
  =
  \frac{\pr\{\tau_x\in\dd t,\, T_c>\tau_x\}}{\dd t}.
\end{align*}
Since $\tau_x$ has a \pdf, $k_{x,c}(t)$ is well-defined, and since its
integral over $t$ is $\pr\{T_c > \tau_x\}<1$, it is a sub-\pdf rather
than a proper one.  For $b>0$ and $c>0$, it is well-known that (\cite
{kyprianou:14:sh}, Theorem 8.1)
\begin{align} \label{e:LT-k}
  \LT{k_{-b,c}}(q) = \frac{W\Sp q(c)}{W\Sp q(b+c)}
  =
  \frac{c^{\alpha-1}}{(b+c)^{\alpha-1}}
  \frac{
    E_{\alpha, \alpha}(c^\alpha q)
  }{
    E_{\alpha, \alpha}((b+c)^\alpha q)
  }.
\end{align}
  
\begin{prop} \label{p:exit-lower}
  Given $b>0$ and $c>0$, put $s = c^\alpha/(b+c)^\alpha$.  Then 
  \begin{align} \label{e:k-psi}
    k_{-b,c}(t)
    =
    \frac{c^{\alpha-1}}{(b+c)^{2\alpha-1}}
    \psi_s\Grp{\frac{t}{(b+c)^\alpha}},
  \end{align}
  where 
  \begin{align} \label{e:residual}
    \psi_s(t)
    = 
    \sum_{\root\in\Scr Z_{\alpha,\alpha}} \res(H_s(z) e^{t z},
    \root), \quad t>0,
  \end{align}
  is a \pdf concentrated on $[0,\infty)$ and for $s\in [0,1)$
  \begin{align}  \label{e:H_s(q)}
    H_s(z) =   \frac{
      E_{\alpha, \alpha}(s z)
    }{
      E_{\alpha, \alpha}(z)
    }, \quad z\in\Coms.
  \end{align}
  Furthermore, $\psi_s\in C^\infty(\Reals)$ such that for all $n\ge
  1$, $\psi\Sp n_s(x)\to 0$ as $x\dto 0$ or $x\toi$.
\end{prop}

\begin{remark}
  Since $\pr\{\tau_{-b} < T_c\} = \pr\{\tau_{-b} < \tau_c\} = 
  \lfrac{c^{\alpha-1}}{(b+c)^{\alpha-1}}$ (\cite {bertoin:96:cup},
  Theorem  VII.8), by \eqref {e:k-psi}, conditional on the event that
  $X$ makes first exit from $[-b,c]$ by hitting $-b$, the scaled exit
  time $(b + c)^\alpha\tau_{-b}$ has \pdf $\psi_s$ with $s
  =c^\alpha/(b+c)^\alpha$.
\end{remark}

The main feature of Proposition \ref{p:exit-lower} is that it
expresses the \pdf of the time of first exit in terms of the roots of
the \mittag function $E_{\alpha,  \alpha}(z)$.  As one may suspect,
the expression results from residual calculus for \eqref{e:LT-k} as a
meromorphic function on $\Coms$.  However, since currently there is
little  precise knowledge on the values of the roots of
$E_{\alpha,\alpha}$, the contour involved in the calculation has to
be chosen carefully.  For each term in the sum \eqref
{e:residual}, if $\root\in \Scr Z_{\alpha, \alpha}$ has multiplicity
$n$, then for $t>0$,
\begin{align*}
  \res(H_s(z) e^{t z}, \root)
  =
  \nth{(n-1)!}
  \lim_{z\to\root}
  \frac{\dd^{n-1}}{\dd z^{n-1}}\Sbr{
    \frac{E_{\alpha, \alpha}(s z) e^{t z}}
    {E_{\alpha, \alpha}(z)/(z-\root)^n}
  },
\end{align*}
which has the form $\sum^{n-1}_{k=0} c_k(\root) t^{n-1-k} e^{t
  \root}$.  Moreover, if $\root$ has large enough modulus, then it
is a  simple root (\cite {popov:13:jmathsci}, Theorem 2.1.1), giving
\begin{align*}
  \res(H_s(z) e^{t z}, \root)
  = 
  \frac{E_{\alpha, \alpha}(s\root) e^{t\root}}
  {E'_{\alpha,\alpha}(\root)}. 
\end{align*}

Proposition \ref {p:exit-lower} is an extension of a similar result on
first exit of a standard Brownian motion.  If $\alpha=2$, then by
$\mean[e^{-q X_t}] = e^{t q^2}$ for $q>0$, $X_t = B_{2t}$ with $B_t$ a
standard Brownian motion.  By $E_{2,2}(z) = \sinh(\sqrt z)/\sqrt z$,
$\Scr  Z_{2,2} = \{-k^2\pi^2, k\in\Nats\}$ and $E'_{2,2}(z) =
[\cosh(\sqrt z)- E_{2,2}(z)]/(2 z)$.  Since $E'_{2,2}(-k^2  \pi^2) =
(-1)^{k-1}/(2k^2\pi^2)$, each root is simple.  Then from the above
display with $s = c^2/(b+c)^2$ and Proposition \ref{p:exit-lower},
\begin{align*}
  \frac{\pr\{\tau_{-b} \in \dd t,\, \tau_c>\tau_{-b}\}}{\dd t}
  =
  \frac{2\pi}{(b+c)^2} \sumoi k (-1)^{k-1} k \sin\Grp{
    \frac{k\pi c}{b+c}
  }
  \exp\Cbr{-\frac{k^2\pi^2 t}{(b+c)^2}}.
\end{align*}
The series expansion is different from the one in \cite
{borodin:02:bvb} (p.~212, 3.0.6).  However, it can be proved using the
heat equation method (\cite{morters:10:cup}, section 7.4);
see for example \cite{chi:16:tr} for a derivation.
  
Following a heuristic for a standard Brownian motion (cf.\ \cite
{morters:10:cup}, p.~217), one can get a different expression of
$k_{-b,c}$ analogous to the one for the standard Brownian motion in
\cite{borodin:02:bvb} (p.~212, 3.0.6):
\begin{align} \nonumber
  k_{-b,c}
  &= 
  f_{-b} - f_c * f_{-b-c} + f_{-b} * f_{b+c} * f_{-b-c} - 
  f_c * f_{-b-c} * f_{b+c} * f_{-b-c} + \cdots
  \\\label{e:k-alt}
  &=
  \sumzi n f_{-b}* (\delta - f_c*f_{-c}) * (f_{b+c} * f_{-b-c})^{*n},
\end{align}
where all the terms involved are taken as \pdf's of time, $\delta$ is
the degenerate distribution at 0, and $p^{*0}:=\delta$ for any \pdf
$p$.  Indeed, thinking of $f_{-b}(t)$ as the probability that $X$ hits
$-b$ at time $t$ for the first time, and $k_{-b,c}(t)$ as the one that
$X$ does so before it ever hits $c$, $f_{-b}(t) - k_{-b,c}(t)$ is the
probability that $X$ does so after it hits $c$, so by strong Markov
property,
\begin{align} \label{e:k-k}
  k_{-b,c}(t) = f_{-b}(t) - (k_{c,-b} * f_{-b-c})(t),
\end{align}
where we have used the extended definition
$k_{x,c}(t)=\lfrac{\pr\{\tau_x \in \dd  t,\, \tau_c > \tau_x\}}{\dd
  t}$ for any $x$, $c\in\Reals$.  Likewise, $k_{c,-b}(t) = f_c(t) -
(k_{-b,c} * f_{b+c})(t)$.  Then iterating the two
identities gives \eqref{e:k-alt}.  A rigorous proof of \eqref {e:k-alt} will be considered
in Section \ref {ss:exit-lower-alt}.

\subsection{Properties of scaled first exit time at lower end}
This subsection considers some properties of $H_s(z)$ as defined in
\eqref{e:H_s(q)}.  Along the way it proves the smoothness of
$\psi_s$ stated at the end of Proposition \ref{p:exit-lower}.

From \eqref {e:LT-k} and scaling, it follows that for any $s\in
(0,1)$, $H_s(q)$ is the Laplace transform of the probability
distribution
\begin{align*}
  \mu_s(\dd t)
  =
  \pr\{\tau_{s^\ralph-1}
  \in\dd t\gv \tau_{s^\ralph - 1} < T_{s^\ralph}\},  \quad
  t>0,
\end{align*}
i.e., for $q>0$,
\begin{align}  \label{e:mu_s}
  \LT{\mu_s}(q) = H_s(q)
  =
  \frac{E_{\alpha,\alpha}(sq)}{E_{\alpha,\alpha}(q)}.
\end{align}
It is seen that the identity holds for all $q\in\Coms$ with $\Re(q)
\ge0$.  By \eqref{e:mu_s}, given $s\in (0,1)$, $H_s(q)$ is completely
monotone in $q\ge 0$, and for $z\in\Coms$ with $\Re(z)\ge0$,
$|H_s(z)|\le1$, and so $|E_{\alpha, \alpha}(s z)|\le |E_{\alpha,
  \alpha} (z)|$.  Thus for any $\theta\in [-\pi/2,\pi/2]$,
$|E_{\alpha, \alpha}(e^{\iunit \theta} r)|$ is increasing in $r\ge0$,
so letting $r\to0+$,
\begin{align*}
  |E_{\alpha, \alpha}(z)|\ge E_{\alpha, \alpha}(0)
  =
  \frac1{\Gamma(\alpha)},
  \quad \Re(z)\ge 0.
\end{align*}

Fix $s\in (0,1)$.  For $\theta\in\Reals$,
\begin{align} \label{e:FT}
  \FT{\mu_s}(\theta) = H_s(-\iunit \theta)
  =
  \frac{E_{\alpha, \alpha}(-\iunit s\theta)}
  {E_{\alpha, \alpha}(-\iunit \theta)}.
\end{align}
On the other hand, $|e^{(-\iunit \theta)^\ralph}| = e^{\lambda
|\theta|^\ralph}$ with $\lambda =\cos(\alpha^{-1} \pi/2)>0$, so by
\eqref{e:asym-infty},
\begin{align} \label{e:asym-infty2}
  |E_{\alpha,   \alpha}(-\iunit\theta)| \sim \alpha^{-1}
  |\theta|^{\ralph-1} e^{\lambda|\theta|^\ralph}, \quad
  \theta\to\pm\infty
\end{align}
where $x\sim y$ means $x/y\to1$.  Then from \eqref{e:FT},
$|\FT{\mu_s}(\theta)| \sim s^{\ralph-1} e^{\lambda(s^\ralph-1) 
  |\theta|^\ralph}$.  As a result, $\int |\FT{\mu_s}(\theta)| 
|\theta|^n \,\dd \theta < \infty$ for all $n\ge 0$, so $\mu_s$ has
a \pdf in $C^\infty(\Reals)$ with vanishing derivative of any order at
$\pm\infty$ (\cite{sato:99:cup}, Proposition 28.1).  By \eqref
{e:LT-k}, the \pdf is exactly $\psi_s$ in Proposition \ref
{p:exit-lower}.  Since $\psi_s$ is supported on $[0,\infty)$, it is
seen that $\psi\Sp n_s(x)\to 0$ as $x\to0+$.

From \eqref{e:FT} and the Continuity Theorem of characteristic
functions (cf.\ \cite {breiman:92:siam}, Theorem 8.28), as $s\to 0+$,
$\mu_s$ weakly converges to a probability distribution $\mu_0$ with
\begin{align*}
  \FT{\mu_0}(\theta) =
  H_0(-\iunit\theta) = 
  \frac1{\Gamma(\alpha)E_{\alpha, \alpha}(-\iunit\theta)},
  \quad\theta\in \Reals.
\end{align*}
Similar to $\mu_s$ with $s\in (0,1)$, $\mu_0$ has a \pdf $\psi_0\in
C^\infty(-\infty, \infty)$ with support on $[0,\infty)$ such that all
its derivatives $\psi\Sp n_0(x)$ vanish as $x\to 0+$, $\infty$.
Consequently, for each $s\in [0,1)$, $\psi_s$ cannot be analytically
extended to a neighborhood of 0, otherwise, $\psi_s$ would be constant
0.  On the other hand, by Fourier inversion (\cite{rudin:87:mcgraw},
p.~185),
\begin{align} \label{e:exit-lower-ift}
  \psi_s(t)
  =
  \nth{2\pi} \intii \FT{\mu_s}(\theta) e^{-\iunit\theta t}\,\dd\theta
  =
  \nth{2\pi} \lim_{M\toi} \int^M_{-M}
  \FT{\mu_s}(\theta) e^{-\iunit\theta t}\,\dd\theta,
\end{align}
and from \eqref{e:mu_s}, $\mu_s$ has finite moment of any order, with
the $n\th$ moment equal to $(-1)^n H\Sp n_s(0)$.

\subsection{Contour integration}
This subsection completes the proof of Proposition \ref{p:exit-lower}.
Because \eqref{e:k-psi} directly follows from \eqref{e:LT-k}, it
only remains to show \eqref{e:residual}.

\begin{proof}[Proof of Eq.~\eqref{e:residual}]
  Define function
  \begin{align*}
    \sigma(\theta) = \nth{|\sin(\theta/\alpha)|}.
  \end{align*}
  Since $\alpha>1$, $\sigma(\theta)$ is bounded on $[-\pi, -\pi/2]\cup
  [\pi/2, \pi]$.  Fix
  $\beta\in (\pi/2, \pi/\alpha)$.  For $R>0$, let $C_R$ be the contour
  that travels along the curve
  \begin{align} \label{e:contour}
    \{[R \sigma(\theta)]^\alpha e^{\iunit\theta}:\,
      \pi/2\le |\theta|\le \pi\}
  \end{align}
  from $\iunit (R\sigma_0)^\alpha$ to $-\iunit (R\sigma_0)^\alpha$,
  where $\sigma_0 = \sigma(\pi/2)$.  The contour is smooth except at its
  intersection with $(-\infty,0)$ and its length is proportional to
  $R^\alpha$.  Fix $\beta \in (\pi/2, \alpha\pi/2)$.  Let $C_{R,1}$
  denote the part of $C_R$ in the section $\pi/2\le |\arg z|\le
  \beta$, and $C_{R,2}$ the part in the section $\beta \le |\arg z|
  \le \pi$.

  For $z = re^{\iunit\theta}$ with $\theta = \arg z$,
  $|\exp(z^\ralph)| = \exp\{r^\ralph \cos(\theta/\alpha)\}$.  If $z\in
  C_{R,1}$, then $|\theta/\alpha|\le \beta/\alpha < \pi/2$, and so
  $\cos(\theta/\alpha)\ge \lambda:=\cos(\beta/\alpha)>0$.  As a result,
  for $z\in C_{R,1}$,
  \begin{align} \label{e:modulus-exp}
    |\exp(z^\ralph)| \ge \exp(\lambda |z|^\ralph).
  \end{align}
  Then by \eqref{e:asym-infty}, given $s\in (0,1)$, as $R\toi$,
  \begin{align*}
    H_s(z)
    = \frac{E_{\alpha, \alpha}(s z)}{E_{\alpha, \alpha}(z)}
    &=
    (1+o(1))
    \frac{(s z)^{\ralph-1} \exp((s z)^\ralph)}
    {z^{\ralph-1} \exp(z^\ralph)}
    \\
    &=
    (1+o(1)) s^{\ralph-1} \exp\{(s^\ralph-1) z^\ralph\}, \quad
    z\in C_{R,1},
  \end{align*}
  where the $o(1)$ term is uniform in $z\in C_{R,1}$.  Since $|z|\ge
  R^\alpha$, from \eqref{e:modulus-exp},
  \begin{align} \label{e:contour1}
    \sup_{z\in C_{R,1}} |H_s(z)| = O(\exp\{-\lambda(1-s^\ralph) R\}).
  \end{align}

  We also need a bound for $|H_s(z)| = |E_{\alpha,\alpha}(s z) /
  E_{\alpha, \alpha}(z)|$ on $C_{R,2}$.  However, since $E_{\alpha,
    \alpha}(z)$ has infinitely many roots in the section $\beta \le
  |\arg z|\le \pi$, $R$ cannot be any large positive number but has to
  be selected appropriately.  We need the following result.  
  \begin{lemma} \label{l:contour}
    Let $R_n = 2\pi n$, $n=1,2,\ldots$.  Then given any $A\in
    \Reals\setminus\{0\}$, 
    \begin{align} \label{e:dist}
      \Linf_{n\toi} \inf_{z\in C_{R_n,2}}
      |z^{\ralph+1} \exp(z^\ralph) - A|>0.
    \end{align}
  \end{lemma}
  
  Assuming the lemma is true for now, let $A = \alpha^2(\alpha-1)/
  \Gamma(2-\alpha)$.  By \eqref{e:asym-infty},
  \begin{align*}
    E_{\alpha, \alpha}(z)
    =
    \alpha^{-1} z^{-2} [z^{\ralph+1} \exp(z^\ralph) - A] +
    O(|z|^{-3}).
  \end{align*}
  Then by Lemma \ref{l:contour}, there is $\rx>0$, such that for all
  large $n$ and $z\in C_{R_n,2}$, $E_{\alpha, \alpha}(z) \ge\rx
  |z|^{-2}$.  Let $m_0 = \sup_{\pi/2\le |\theta|\le\pi}
  \sigma(\theta)$.  Then by $|z|\le (m_0 R_n)^\alpha$,
  \begin{align*}
    E_{\alpha,  \alpha}(z) \ge\rx m^{-2\alpha}_0 R^{-2\alpha}_n.
  \end{align*}
  On the other hand, again by \eqref {e:asym-infty},
  \begin{align*}
    |E_{\alpha,
      \alpha}(s z)|\le
    E_{\alpha, \alpha}(|z|) \le E_{\alpha, \alpha}(m^\alpha_0
    R^\alpha_n) = O(R^{1-\alpha}_n \exp(m_0 R_n)).
  \end{align*}
  Combining with the lower bound, this implies 
  \begin{align} \label{e:contour2}
    \sup_{z\in C_{R_n,2}} |H_s(z)|
    = O(R^{1+\alpha}_n e^{m_0 R_n}), \quad n\toi.
  \end{align}
  
  Let $D_R$ be the domain bounded by $C_R$ and $\{\iunit\theta:
  |\theta| \le (R\sigma_0)^\alpha\}$.  Let $t>0$.  If $C_R\cap \Scr
  Z_{\alpha,\alpha} = \emptyset$, then by \eqref{e:FT}
  and residual theorem,
  \begin{align} \nonumber
    \nth{2\pi}
    \int^{(R\sigma_0)^\alpha}_{-(R\sigma_0)^\alpha}
    \FT{\mu_s}(\theta) e^{-\iunit\theta t}\,\dd\theta 
    &=
    \nth{2\pi\iunit}
    \int^{\iunit (R\sigma_0)^\alpha}_{-\iunit (R\sigma_0)^\alpha}
    H_s(z) e^{t z}\,\dd z
    \\\label{e:ift}
    &=
    \nth{2\pi\iunit}
    \int_{C_R} H_s(z) e^{t z}\,\dd z -
    \sum_{\root\in D_R\cap \Scr Z_{\alpha,\alpha}}
    \res(H_s(z) e^{t z}, \root).
  \end{align}
  Consider the contour integral along $C_R$.  For $z = r
  e^{\iunit\theta}\in C_R$ with $\theta = \arg z$, by $\pi/2\le
  |\theta|\le \pi$, $|e^{t z}|= e^{r t \cos\theta} \le 1$.  Then by
  \eqref{e:contour1},
  \begin{align*}
    \Abs{\int_{C_{R,1}} H_s(z) e^{t z}\,\dd z}
    &\le
    \text{Length}(C_{R,1})\times O(e^{-\lambda(1-s^\ralph) R})
    \\
    &=
    O(1) R^\alpha e^{-\lambda(1-s^\ralph) R}.
  \end{align*}
  Furthermore, if $z\in C_{R,2}$, then by $\beta\le |\theta| \le\pi$
  and $r\ge R^\alpha$, $|e^{t z}| \le e^{- b_0 R^\alpha t}$, where
  $b_0 = - \cos\beta>0$.  Then by \eqref{e:contour2},
  \begin{align} \nonumber
    \Abs{\int_{C_{R_n,2}} H_s(z) e^{t z}\,\dd z}
    &\le
    \text{Length}(C_{R_n,2}) \times O(R^{1+\alpha}_n e^{m_0 R_n - b_0
      R^\alpha_n t})
    \\\label{e:contour3}
    &=
    O(1) R^{1+2\alpha}_n e^{m_0 R_n - b_0 R^\alpha_n t}.
  \end{align}
  By $\alpha>1$, combining the above two bounds yields
  \begin{align*}
    \int_{C_{R_n}} H_s(z) e^{t z}\,\dd z\to0.
  \end{align*}
  Then by the Fourier inversion \eqref{e:exit-lower-ift} and
  \eqref{e:ift},
  \begin{align*}
    \psi_s(t) =
    \lim_{n\toi}
    \sum_{\root\in D_{R_n}\cap \Scr Z_{\alpha,\alpha}}
    \res(H_s(z) e^{t z}, \root), \quad t>0.
  \end{align*}
  
  The above formula is proved for $s\in (0,1)$.  For $s=0$, the
  formula can be similarly proved.  To complete the proof, it only
  remains to show that the series on the \rhs of
  \eqref{e:residual} converges absolutely.  Since in any bounded
  domain there are only a finite number of roots of $E_{\alpha,
    \alpha}(z)$, it suffices to show that for a large enough $M>0$, 
  \begin{align*}
    \sum_{\root\in \Scr Z_{\alpha,\alpha}\setminus U_M(0)}
    |\res(H_s(z) e^{t z}, \root)|<\infty.
  \end{align*}
  
  By Theorems 2.1.1 and Chapter 6 of \cite {popov:13:jmathsci}, $M>0$
  can be chosen such that all elements in $\Scr Z_{\alpha,\alpha}
  \setminus U_M(0)$ are not real and are simple roots of $E_{\alpha,
    \alpha}$.  Then for each $\root\in \Scr Z_{\alpha,
    \alpha}\setminus U_M(0)$,
  \begin{align*}
    \res(H_s(z) e^{t z}, \root)
    =
    \res\Grp{
      \frac{E_{\alpha, \alpha}(s z) e^{t z}}{E_{\alpha, \alpha}(z)},
      \root
    }
    =
    \frac{E_{\alpha, \alpha}(s\root) e^{t\root}}
    {E'_{\alpha,\alpha}(\root)}. 
  \end{align*}
  Letting $z = \root$ in the following identity (\cite
  {popov:13:jmathsci}, p.~333)
  \begin{align*}
    \alpha z E'_{\alpha, \alpha}(z) + (\alpha-1) E_{\alpha, \alpha}(z)
    =
    E_{\alpha, \alpha-1}(z),
  \end{align*}
  one gets $E'_{\alpha, \alpha}(\root) = E_{\alpha,
    \alpha-1}(\root)/(\alpha\root)$ and hence
  \begin{align} \label{e:ML-recursive}
    \res(H_s(z) e^{t z}, \root)
    = 
    \frac{\alpha \root E_{\alpha, \alpha}(s\root)e^{t\root}}
    {E_{\alpha,\alpha-1}(\root)}.
  \end{align}
  By $E_{\alpha,\alpha}(\root)=0$ and \eqref{e:asym-infty},
  \begin{align*}
    \alpha^{-1} \root^{\ralph-1} \exp(\root^\ralph)
    = 
    \frac{\alpha(\alpha-1) \root^{-2}}{\Gamma(2-\alpha)}
    + O(|\root|^{-3}).
  \end{align*}
  On the other hand, by Theorems 1.2.1, 1.4.2, and 1.5.1 in
  \cite{popov:13:jmathsci},
  \begin{align*}
    E_{\alpha,\alpha-1}(\root)
    =
    \alpha^{-1} \root^{2/\alpha-1} \exp(\root^\ralph) +
    \frac{\alpha(\alpha^2-1)\root^{-2}}{\Gamma(2-\alpha)}
    + O(|\root|^{-3}).
  \end{align*}
  Combining the two displays yields $E_{\alpha,\alpha-1}(\root) \asymp
  \root^{\ralph-2}$.  Then by \eqref{e:ML-recursive},
  \begin{align*}
    |\res(H_s(z) e^{t z},\root)|
    =
    O(1) |\root^{3-\ralph} E_{\alpha, \alpha}(s \root) e^{t\root}|.
  \end{align*}
  From \eqref{e:asym-infty}, as $|\root|\toi$,
  \begin{align*}
    |E_{\alpha,\alpha}(s
    \root)| \le E_{\alpha, \alpha}(|\root|) = O(1) |\root|^{\ralph-1}
    \exp(|\root|^\ralph).
  \end{align*}
  On the other hand, let $\root = r e^{\iunit \theta}$ with $\theta =
  \arg\root$.  By $\alpha\pi/2 < |\theta|\le \pi$, $|e^{\root t}| =
  e^{t r\cos\theta} \le \exp(-\lambda |\root| t)$, where $\lambda =
  -\cos(\alpha\pi/2)>0$.  Then
  \begin{align*}
    |\res(H_s(z) e^{t z},\root)|
    =
    O(1) |\root|^2 e^{|\root|^\ralph - \lambda t |\root|}
    =
    O(1) e^{-\eta|\root|}
  \end{align*}
  for some $\eta = \eta(t)>0$.  By \eqref{e:zero-asym}, for $M>0$
  large enough, all $\root\in \Scr Z_{\alpha,\alpha} \setminus U_M(0)$
  can be enumerated as $\root_{\pm n}$, $n=N$, $N+1$, \ldots, with
  $N\ge1$ being some large integer, such that $|\root_{\pm n}| \asymp
  n^\alpha$.   This combined with the above display then yields the
  desired absolute convergence.
\end{proof}

\begin{proof}[Proof of Lemma \ref{l:contour}]
  For $z = [R \sigma(\theta)]^\alpha e^{\iunit\theta}\in C_{R,2}$ with
  $\theta = \arg z$,
  \begin{align*}
    z^{\ralph+1} \exp(z^\ralph)
    &=
    [R \sigma(\theta)]^{1+\alpha} e^{\iunit(\ralph+1)\theta} 
    \exp\{R \sigma(\theta) e^{\iunit\theta/\alpha}\}
    \\
    &=
    [R \sigma(\theta)]^{1+\alpha} e^{R \sigma(\theta)
      \cos(\theta/\alpha)}
    e^{\iunit[(\ralph+1)\theta + R \sigma(\theta) \sin(\theta/\alpha)]}
    \\
    &=
    [R \sigma(\theta)]^{1+\alpha} e^{R \sigma(\theta)
      \cos(\theta/\alpha)}
    e^{\iunit[(\ralph+1)\theta + R\text{sign}(\theta)]}.
  \end{align*}
  Put $a(\theta, R)=[R \sigma(\theta)]^{1+\alpha} e^{R \sigma(\theta)
    \cos(\theta/\alpha)}$.  Then for $z\in C_{R_n,2}$, by $R_n = 2\pi
  n$, $z^{\ralph+1} \exp(z^\ralph) = a(\theta, R_n)  e^{\iunit(\ralph
    + 1)\theta}$.  If there were $z_n = [R \sigma(\theta_n)]^\alpha
  e^{\iunit \theta_n} \in C_{R_n,2}$ such that  $z^{\ralph + 1}_n
  \exp(z^\ralph_n) \to A$, then taking modulus, $a(\theta_n, R_n) =
  [R_n \sigma(\theta_n)]^{1 + \alpha} e^{R_n \sigma(\theta_n)
    \cos(\theta_n/\alpha)} \to |A|>0$.  By $|R_n \sigma(\theta_n)|
  \toi$, it follows that $\cos(\theta_n/\alpha)\to 0$, as any
  sequence $n$ with $R_n \sigma(\theta_n) \cos(\theta_n/\alpha)\toi$
  (resp.\ $-\infty$) has $a(\theta_n, R_n)\toi$ (resp.\ 0).
  Because $|\theta_n|/\alpha \in (\pi/(2\alpha), \pi/\alpha]$, this
  implies $\theta_n/\alpha = k_n\pi/2 + \rx_n$ with $k_n = \pm1$ and
  $\rx_n\to0$.  But then
  \begin{align*}
    z^{\ralph+1}_n \exp(z^\ralph_n)  =
    a(\theta_n, R_n) e^{\iunit (\ralph + 1) \theta_n} = |A|e^{\iunit(1 +
      \alpha) k_n\pi/2} + o(1)\not\to A,
  \end{align*}
  a contradiction. 
\end{proof}

\subsection{Alternative expression and asymptotics}
\label{ss:exit-lower-alt}
We first consider \eqref{e:k-alt}.  Let
\begin{align*}
  M(t) = \sumzi n f_{-b}*(\delta + f_c
  *f_{-c})*(f_{b+c}*f_{-b-c})^{*n}(t).
\end{align*}
Given $q>0$, by Fubini's theorem, $\LT M(q) =\LT{f_{-b}}(q) [1 +
\LT{f_c}(q) \LT{f_{-c}}(q)] \sumzi n \LT{f_{b+c}}(q)^n \LT{f_{-b
    -c}}(q)^n$, which is finite due to $\LT{f_{-b-c}}(q) < 1$.  Then
$M(t)<\infty$ a.e.  As a result, the \rhs of \eqref{e:k-alt} converges
a.e.  Denote it by $F(t)$  for now.  By dominated convergence and the
formula for $\LT M(q)$,
\begin{align*}
  \LT F(q)
  =
  \frac{\LT{f_{-b}}(q) [1- \LT{f_c}(q)  \LT{f_{-c}}(q)]}
  {1-\LT{f_{b+c}}(q) \LT{f_{-b-c}}(q)}.
\end{align*}
On the other hand, it is known that \cite{doney:91:jlms} (also cf.\
\cite{kyprianou:14:sh}, p.~253) 
\begin{align*}
  \LT{f_x}(q) = e^{x q^\ralph} - \alpha x^{\alpha-1}_+ q^{1 - \ralph}
  E_{\alpha, \alpha}(x^\alpha_+ q), \quad x\in\Reals,\ q>0.
\end{align*}
Then for $x\ge0$, $\LT f_{-x}(q) = e^{-x q^\ralph}$
and $1 - \LT f_x(q) \LT f_{-x}(q) = \alpha x^{\alpha-1} q^{1-\ralph} 
E_{\alpha,\alpha}(x^\alpha q) e^{-x q^\ralph}$.  Plugging the
identities into the display, it is seen that $\LT F(q)$ equals the
\rhs of \eqref{e:LT-k}, giving $F(t) = k_{-b,c}(t)$.

Based on the alternative expression, it is quite easy to get that as
$t\dto0$, 
\begin{align} \label{e:asym-k-0}
  k_{-b,c}(t) \sim f_{-b}(t),
\end{align}
in particular, by Eq.~(14.35) in \cite{sato:99:cup}, $\ln k_{-b,c}(t)
\sim -C b^{\alpha/(\alpha-1)} t^{-1/(\alpha-1)}$, where $C>0$ is
constant.  First, by \eqref{e:k-k}, $0<f_{-b}(t) - k_{-b,c}(t) =
(k_{c,-b}*f_{-b - c})(t) < (f_c*f_{-b-c})(t) = (u*f_{-b})(t)$, where
$u = f_c*f_{-c}$ is a \pdf and we have used $k_{c,-b} < f_c$ and
$f_{-b-c} = f_{-b} * f_{-c}$.  Next,
\begin{align*}
  (u*f_{-b})(t)
  =
  \int^t_0 u(s) f_{-b}(t-s)\,\dd s
  \le
  \sup_{s\le t} f_{-b} \times \int^t_0 u
  =
  o(1) \sup_{s\le t} f_{-b}, \quad t\dto 0.
\end{align*}
Since $f_{-b}$ is unimodal (\cite{sato:99:cup}, p.~416), for all $t>0$
small enough, $\sup_{s\le t} f_{-b} = f_{-b}(t)$, implying
\eqref{e:asym-k-0}.

Although Proposition \ref{p:exit-lower} gives a series expression of
$\mu_s$, it does not provide the radius of convergence of
$\LT{\mu_s}$, defined as $\sup\{r>0: e_z\in L^1(\mu_s)\,\forall z\in
U_r(0)\}$, where $e_z$ is the function $t\mapsto e^{-t z}$.  This is
also related to the tail of $k_{-b,c}(t)$ as  $t\toi$.  We have
the following.

\begin{prop} \label{p:root}
  Let $-\varrho$ be the largest real root of $E_{\alpha,
    \alpha}(z)$, where $\varrho>0$.  Then given $s\in [0,1)$,
  $\LT{\mu_s}(x) < \infty$ for $x\in(-\varrho, \infty)$ and
  $\LT{\mu_s}(x)\uto \LT{\mu_s}(-\varrho) = \infty$ as $x\dto
  -\varrho$.  In particular, the radius of convergence of $\LT{\mu_s}$
  is $\varrho$ and $\Scr Z_{\alpha,\alpha} \subset \{z: |\arg z| >
  \alpha\pi/2,\, \Re(z)\le -\varrho\}$.
\end{prop}
\begin{proof}
  Recall that for a measure $\nu$ on $[0,\infty)$ with finite total
  mass, the domain of $\LT{\mu_s}$ contains $\{z: \Re(z)\ge 0\}$, and
  if $\LT\nu(z)$ can be analytically extended to $U_r(0)\cap \{z:
  \Re(z) < 0\}$ for some $r>0$, then the domain of $\LT\nu$ contains
  $\{z: \Re(z)>-r\}$ and $\LT\nu$ is analytic in the region.

  Let $\root_0$ be a root of $E_{\alpha,\alpha}(z)$ with the largest
  real part.  Put $a = -\Re(\root_0)$.  Clearly, $\varrho \ge a$.
  From $\Scr Z_{\alpha,\alpha} \subset \{z: |\arg(z)| >
  \alpha\pi/2\}$ (\cite{popov:13:jmathsci}, Theorem 4.2.1),
  $a>0$.  Since $E_{\alpha, \alpha}(z) \ne0$ for any $z$ with
  $\Re(z)>-a$, $H_s(z)$ is analytic in $U_a(0)$.  Since, by
  \eqref{e:H_s(q)} and \eqref  {e:mu_s}, $\LT{\mu_s}(z) = H_s(z)$ for
  $z$ with $\Re(z) \ge0$, from the above remark, the domain of
  $\LT{\mu_s}$ contains $\{z: \Re(z) >
  -a\}$.  Let $z\to\root_0$ along the ray from 0 to $\root_0$. Then
  $|\LT{\mu_s}(z)| = |H_s(z)| \to |H_s(\root_0)|= \infty$.  By
  $|\LT{\mu_s}(z)| \le \LT{\mu_s}(\Re(z))$, it follows that if $x\in
  \Reals$ and $x\dto -a$, then $\LT{\mu_s}(x) = H_s(x)\toi$.  Since
  $E_{\alpha,\alpha}(x) \to E_{\alpha,\alpha}(-sa)>0$, then
  $E_{\alpha, \alpha}(-a)=0$, and so $a = \varrho$.  Thus, $\Scr
  Z_{\alpha, \alpha} \subset \{z: \Re(z)\le -\varrho\}$.  The proof is
  complete.
\end{proof}

Combining Propositions \ref{p:exit-lower} and \ref{p:root}, if the
multiplicity of $-\varrho$ is $n\ge 1$, then $k_{-b,c}(t)$ decreases
exponentially fast with
\begin{align*}
  \limsup_{t\toi}
  \frac{\ln k_{-b,c}(t)}{t} = -\frac{\varrho}{(b+c)^\alpha}.
\end{align*}
However, the exact asymptotic of
$k_{-b,c}(t)$ depends on the multiplicity of $-\varrho$ as well as
other roots of $E_{\alpha, \alpha}(z)$ on the line $\Re(z) =
-\varrho$, if there are any.

\section{Distribution of first exit at upper end} \label{s:upper}
The main result of this section is Theorem \ref{t:exit-upper} below.
It provides a factorization of the joint sub-\pdf of the time
$T_c$, the undershoot $X_{T_c-}$, and the jump $\Delta_{T_c} = X_{T_c}
- X_{T_c-}$ when $X$ makes first exit from $[-b,c]$ by jumping upward
across $c$ before hitting $-b$.  The following function plays an
important role.  For $x\in (-b,c)$ and $t>0$, define
\begin{align} \label{e:no-exit}
  l_{x,-b,c}(t)
  =
  \frac{\pr\{X_t\in\dd x,\, X_s\in (-b,c)\,\forall s\le t\}}{\dd x}.
\end{align}
While the function can be defined for any process that has a \pdf at
any time point, in the case of a spectrally positive strictly stable
process, it has an explicit representation. 

To start with, it is known that for $q\ge 0$ (\cite{kyprianou:14:sh},
Theorem 8.7)
\begin{align} \nonumber 
  \LT{l_{x, -b, c}}(q)
  &=
  \frac{W\Sp q(c) W\Sp q(b+x)}{W\Sp q(b+c)} - W\Sp q(x_+)
  \\\label{e:LT-l}
  &=
  \frac{c^{\alpha-1} (b+x)^{\alpha-1}}{(b+c)^{\alpha-1}}
  \frac{
    E_{\alpha, \alpha}(c^\alpha q)
    E_{\alpha, \alpha}((b+x)^\alpha q)
  }{
    E_{\alpha, \alpha}((b+c)^\alpha q)
  }
  - x^{\alpha-1}_+ E_{\alpha,\alpha}(x^\alpha_+ q).
\end{align}

\begin{theorem} \label{t:exit-upper}
  Fix $b>0$ and $c>0$.  For $x\in \Reals$,
  \begin{align} \label{e:undershoot}
    \pr\{T_c<\tau_{-b},\, X_{T_c-} \in\dd x\}
    =
    \cf{x\in (-b,c)}
    \frac{|\sin(\alpha\pi)|}\pi
    \Sbr{
      \frac{c^{\alpha-1} (b+x)^{\alpha-1}}{(b+c)^{\alpha-1}}
      - x^{\alpha-1}_+ 
    }  \frac{\dd x}{(c-x)^\alpha},
  \end{align}
  and conditional on $T_c < \tau_{-b}$ and $X_{T_c-} = x\in (-b,c)$,
  $\Delta_{T_c}$ and $T_c$ are independent, such that $\Delta_{T_c}$
  follows the Pareto distribution with \pdf
  \begin{align*}
    \pi(u)
    =
    \alpha (c-x)^\alpha u^{-\alpha-1} \cf{u>c-x}
  \end{align*}
  and $T_c$ has \pdf
  \begin{align} \label{e:exit-time-upper}
    p(t) = \Gamma(\alpha)\Sbr{
      \frac{c^{\alpha-1} (b+x)^{\alpha-1}}{(b+c)^{\alpha-1}}
      - x^{\alpha-1}_+ 
    }^{-1} l_{x, -b,c}(t)
  \end{align}
  with $l_{x,-b,c}(t)$ having the following expression
  \begin{align} \label{e:residual2}
    l_{x,-b,c}(t)
    =
    \frac{c^{\alpha-1} (b+x)^{\alpha-1}}{(b+c)^{\alpha-1}}
    \sum_{\root\in \Scr Z_{\alpha,\alpha}} 
    \res\Grp{
      \frac{
        E_{\alpha,\alpha}(c^\alpha z) E_{\alpha,\alpha}((b+x)^\alpha
        z)
      }{
        E_{\alpha,\alpha}((b+c)^\alpha z)
      } e^{t z},\, \frac\root{(b+c)^\alpha}
    }.
  \end{align}
  Furthermore, given $t>0$, the mapping $x\to l_{x,-b,c}(t)$ has an
  analytic extension from $(-b,c)$ to $\Coms\setminus (-\infty, -b]$.
\end{theorem}

From \eqref{e:LT-l}, one may suspect that \eqref{e:residual2} can be
obtained by residual calculus.  However, when $x>0$, the function in
the residuals of \eqref{e:residual2} have the term $x^{\alpha-1}
E_{\alpha, \alpha}(x^\alpha q)$ missing.  A direct inversion of the
Laplace transform \eqref{e:LT-l} when $x>0$ seems to be involved.
Instead, we will first show \eqref{e:residual2} for $x<0$ by residual
calculus, and then establish \eqref{e:residual2} for $x\ge0$ through
analytic extension.

\subsection{Factorization and conditional independence}
\label{ss:factor}
The factorization in Theorem \ref{t:exit-upper} follows from the
next result, which actually holds under much more general assumptions
on a \levy process.
\begin{lemma} \label{l:first-exit}
  Given $b>0$ and $c>0$, for $t>0$,
  \begin{multline} \label{e:joint-exit-up}
    \pr\{T_c<\tau_{-b},\, T_c\in\dd t,\, X_{T_c-}\in\dd x,\,
    \Delta_{T_c}\in \dd u\}
    \\
    =
    \cf{x>-b,\, u>c-x>0}
    \dd t \, \pr\{X_t\in\dd x,\, X_s \in (-b,c)\,\forall s\le t\}
    \,\Pi(\dd u).
  \end{multline}
\end{lemma}
\begin{proof}
  The proof follows the one on p.~76 of \cite{bertoin:96:cup}.  As
  already noted, $\pr\{X_{T_c}>c > X_{T_c-}>-b\}=1$.  By Rogozin's
  criterion (cf.\ \cite {doney:07:sg-b}, Theorem 5.17), under the law
  of $X$, 0 is regular for $(0,\infty)$ as well as for $(-\infty, 0)$,
  so $X_t<c$ for all $t<T_c$ and $X_t>-b$ for all $t<\tau_{-b}$.
  Therefore, almost surely, for any bounded function $f(t,x,u)\ge 0$,
  \begin{align*}
    &
    f(T_c,\, X_{T_c-}, \Delta_{T_c})\cf{T_c < \tau_{-b}}
    \\
    &=
    \sum_t f(t,\, X_{t-},\, \Delta_t)\cf{\Delta_t > c-X_{t-} > 0,\,
      X_{t-}>-b, \, -b<X_s<c\,\forall s\le t}.
  \end{align*}
  The sum on the \rhs is well-defined as it runs over the set of $t$'s
  where $X$ has a jump, which is countable.  The rest of the proof
  then applies the compensation formula to show that the expectation
  of the random sum is an integral of $f(t, x, u)$ with respect to the
  measure on the \rhs of \eqref{e:joint-exit-up}.  Since the argument
  has become standard, it is omitted for brevity.
\end{proof}

Since the \levy measure of $X$ is
\begin{align*}
  \Pi(\dd x) = \frac{\alpha(\alpha-1) \cf{x>0}}
  {\Gamma(2-\alpha)} \frac{\dd x}{x^{\alpha+1}}.
\end{align*}
by \eqref{e:no-exit} and Lemma \ref{l:first-exit}, 
\begin{align}\nonumber
  &
  \pr\{T_c<\tau_{-b},\, T_c\in\dd t,\, X_{T_c-}\in\dd x,\,
  \Delta_{T_c}\in \dd u\}
  \\\label{e:l-factor}
  &=
  \cf{c>x>-b} 
  l_{-x,b,c}(t)\,\dd t\,\dd x\,
  \frac{\cf{u>c-x}\alpha(\alpha-1)\dd
    u}{\Gamma(2-\alpha)u^{\alpha+1}}.
\end{align}
Now let $q=0$ in  \eqref{e:LT-l}.  Then
\begin{align*}
  \intzi l_{x,-b,c}(t)\,\dd t
  =
  \LT{l_{x,-b,c}(0)} 
  =
  \frac1{\Gamma(\alpha)}
  \Sbr{\frac{c^{\alpha-1} (b+x)^{\alpha-1}}{(b+c)^{\alpha-1}}
    - x^{\alpha-1}_+
  }.
\end{align*}
Then by \eqref{e:l-factor}, for $x\in (-b,c)$,
\begin{align*}
  \pr\{T_c<\tau_{-b},\, X_{T_c-}\in\dd x\}
  &=
  \dd x\intzi l_{x,-b,c}(t)\,\dd t \int^\infty_{c-x}
  \frac{\alpha(\alpha-1)}
  {\Gamma(2-\alpha)} \frac{\dd u}{u^{\alpha+1}}
  \\
  &=
  \Sbr{
    \frac{c^{\alpha-1} (b+x)^{\alpha-1}}{(b+c)^{\alpha-1}}
    - x^{\alpha-1}_+ 
  }
  \frac{(\alpha-1)}{\Gamma(\alpha)\Gamma(2-\alpha)}
  \frac{\dd x}{(c-x)^\alpha}.
\end{align*}
Since $\pr\{X_{T_c-}\in (-b,c)\}=1$ by Lemma \ref{l:first-exit}, then
the sub-\pdf formula in \eqref{e:undershoot} follows.  On the other
hand, given $x\in (-b,c)$, 
\begin{align*}
  \pr\{T_c\in\dd t,\, \Delta_{T_c}\in \dd u \gv T_c<\tau_{-b}, \,
  X_{T_c-}\in\dd x\}
  =
  C l_{-x,b,c}(t) \,\dd t \times \frac{\cf{u>c-x}\dd u}{u^{\alpha+1}},
\end{align*}
for some constant $C = C(x)$.  It follows that conditional on $T_c <
\tau_{-b}$ and $X_{T_c-} = x$, $T_c$ and $\Delta_{T_c}$ are
independent, with $\Delta_{T_c}$ following a Pareto distribution and
$T_c$ having a \pdf of the form $C' l_{-x,b,c}(t)$ with $C'$ another
constant, as claimed in Theorem \ref{t:exit-upper}.

\subsection{Contour integration}
The main step is to get the expression of $l_{x,-b,c}(t)$ for given
$x\in (-b,c)$.  As marked after Theorem \ref{t:exit-upper}, we need to
show that $l_{x,-b,c}(t)$ as a function of $x$ can be analytically
extended from $x\in (-b,0)$ to the entire $(-b,c)$.  To this end,
define
\begin{align} \label{e:no-overshoot}
  h_{x,c}(t)
  =
  \frac{\pr\{X_t\in \dd x,\, X_s < c\,\forall s\le t\}}{\dd x}.
\end{align}
Given $c>0$, $h_{x,c}(t)$ can be either regarded as a function of $t$
or a function of $x$.  It already plays a critical role in
\cite{chi:18:tr} in the derivation of the distribution of the triple
$(T_c, X_{T_c-}, X_{T_c})$, known as Gerber-Shiu distribution (\cite
{kyprianou:14:sh}, Chapter 10).

\begin{lemma} \label{l:lhk-renewal}
  Given $b>0$, $c>0$, and $x\in (-b,c)$, $l_{x,-b,c}(t) = h_{x,c}(t) -
  (k_{-b,c}*h_{b+x, b+c})(t)$.
\end{lemma}
\begin{proof}
  Let $f\ge 0$ be a function with support in $(-b,c)$.  Then for any
  $t>0$, 
  \begin{align*}
    \mean [f(X_t) \cf{X_s\in (-b,c)\forall s\le t}]
    =
    \int^c_{-b} f(x) l_{x,-b,c}(t)\,\dd x.
  \end{align*}
  On the other hand, the expectation can be decomposed as
  \begin{align*}
    \mean [f(X_t) \cf{X_s<c\forall s\le t}]
    -
    \mean [f(X_t) \cf{\tau_{-b}\le t,\, X_s<c\forall s\le t}].
  \end{align*}
  The first expectation in the display is equal to
  \begin{align*}
    \int^c_{-b} f(x) h_{x,c}(t)\,\dd x.
  \end{align*}
  By strong Markov property of $X$, the second expectation is equal to
  \begin{align*}
    &
    \int^c_{-b} \int^t_0 f(x) \pr\{X_t\in\dd x,\, \tau_{-b}\in\dd u,\, 
    X_s < c\forall s\le t\}
    \\
    &=
    \int^c_{-b} \int^t_0 f(x) \pr\{X_t\in\dd x+b,\, X_s<b+c\forall
    s\le t-u\} \pr\{\tau_{-b}\in\dd u,\, X_s < c\forall s\le u\}
    \\
    &=
    \int^c_{-b} f(x) \Sbr{
      \int^t_0 h_{x+b,b+c}(t-u) k_{-b,c}(u)\,\dd u
    }\,\dd x.
  \end{align*}
  Comparing the integrals and by $f$ being arbitrary, the claimed
  identity follows.
\end{proof}

\begin{lemma} \label{e:l-extend}
  Given $b>0$, $c>0$, and $t>0$, the mapping $x\to l_{x,-b,c}(t)$ has
  an analytic extension from $(-b,c)$ to $\{z-b: z\in\Omega\}
  \cap\{c-z: z\in\Omega\}$, where $\Omega = \{z\in\Coms: |\arg z| <
  \kappa^{-1}\pi/2\}$ with $1/\kappa = 1 - \ralph$.
\end{lemma}
\begin{proof}
  It is shown in \cite{chi:18:tr} that given $c>0$ and $t>0$, the
  mapping $x\to h_{x,c}(t)$ has an analytic extension from $(-\infty,
  c)$ to $\{c-z: z\in \Omega\}$.  Put $a = b+c$.  Then by Lemma
  \ref{l:lhk-renewal}, it suffices to show that $x\to (k_{-b,c}*h_{x,
    a})(t)$ has an analytic extension from $(0, a)$ to $\Omega\cap
  \{a-z: z\in\Omega\}$.

  For $x < a$, by p.~4/10 of \cite {michna:15:ecp},
  \begin{align*}
    h_{x,a}(t) = g_t(x) - \phi_t(x) > 0,
  \end{align*}
  with
  \begin{align*}
    \phi_t(x) = \int^t_0 f_{x-a}(t-s) g_s(a)\,\dd s>0.
  \end{align*}
  Then $(k_{-b,c}*h_{x,a})(t) = G(x) - \Phi(x)$, where
  \begin{gather*}
    G(x) = \int^t_0 k_{-b,c}(t-s) g_s(x)\,\dd s
    \intertext{and}
    \Phi(x) = \int^t_0 k_{-b,c}(t-s) \phi_s(x)\,\dd s.
  \end{gather*}
  We shall show that $G(x)$ has an analytic extension from $(0,
  \infty)$ to $\Omega$ and $\Phi(x)$ has an analytic extension from
  $(-\infty, a)$ to $\{a-z: z\in\Omega\}$.  This will finish the
  proof.

  First consider $\Phi(x)$.  In \cite{chi:18:tr}, the proof of its
  Lemma 9 establishes that $\phi_t$ has an analytic extension from
  $(-\infty, a)$ to $\{a-z: z\in\Omega\}$ such that, given $0<r_1
  < r_2 < \infty$ and $0< \beta_0 < \kappa^{-1}\pi/2$, for $z =
  re^{\iunit\beta}\in \Omega$ with $r\in [r_1, r_2]$ and $-\beta_0
  \le\beta = \arg z\le \beta_0$, letting $x_0 = r_1 d(\beta_0)>0$,
  where $d(\beta) = \cos(\kappa\beta)^{1/\kappa}$,
  \begin{align*}
    |\phi_t(a-z)|
    \le
    (r_2/x_0)^{\alpha/(\alpha-1)} g_t(a-x_0)<\infty.
  \end{align*}
  By scaling, $g_t(a-x_0) \le t^{-\ralph} \sup g_1$.  By \eqref
  {e:k-psi}, $\sup k_{-b,c}<\infty$.  Then by
  \begin{align*}
    \int^t_0 k_{-b,c}(t-s) |\phi_s(a-z)|\,\dd s
    \le
    \sup k_{-b,c} \int^t_0 (r_2/x_0)^{\alpha/(\alpha-1)}
    s^{-\ralph} \sup g_1\,\dd s,
  \end{align*}
  the \lhs is uniformly bounded for $z$ in any compact subset of
  $\Omega$.  By dominated convergence, $\Phi(a-z)$ is continuous in
  $\Omega$.  Then by Fubini's theorem followed by Morera's theorem
  (cf.\ \cite{rudin:87:mcgraw}, p.~208), $\Phi(a-z)$ is analytic in
  $\Omega$, as desired.

  Next consider $G(x)$.  It is known that $x\to g_t(x)$ has an analytic 
  extension from $\Reals$ to $\Coms$ (\cite{sato:99:cup}, p.~88).
  Then, as above, it suffices to show $\int^t_0 k_{-b,c}(t-s)|g_s(z)|
  \,\dd s$, or more simply, $\int^t_0 |g_s(z)|\,\dd s$  is uniformly
  bounded for $z$ in any compact subset of $\Omega$.  By \cite
  {zolotarev:66:stmsp}, for $x>0$,
  \begin{align*}
    g_t(x)
    =
    \frac{\alpha\pi}{\alpha-1} (x/t)^{\kappa/\alpha}
    \int^{\pi/2}_{\pi(\ralph-1/2)} a(\theta)
    \exp\{-x^\kappa t^{-\kappa/\alpha}
    a(\theta)\}\,\dd\theta,
  \end{align*}
  where $a(\theta)>0$ is a continuous function in the open interval
  of the integral.  Since $z\to z^\kappa$ maps $\Omega$ to $\{z:
  \Re(z)>0\}$, the integral on the \rhs can be analytically extended
  to $\Omega$, and hence the identity can be extended to all
  $z\in\Omega$.   If $z = re^{\iunit\beta}$ with $\beta = \arg z$, then
  by $|\beta|< \kappa^{-1}\pi / 2$, $\Re(z^\kappa) = [r
  d(\beta)]^\kappa>0$, where $d(\beta)$ is defined above.  Then
  \begin{align*}
    |g_t(z)|
    &\le
    \frac{\alpha\pi}{\alpha-1} (r/t)^{\kappa/\alpha}
    \int^{\pi/2}_{\pi(1/\alpha-1/2)} a(\theta)
    \exp\{-[r d(\beta)]^{\kappa} t^{-\kappa/\alpha}
    a(\theta)\}\,\dd\theta
    \\
    &=
    \frac{g_t(r d(\beta))}{d(\beta)^{\kappa/\alpha}}
    \le
    \frac{t^{-\ralph} \sup g_1}{d(\beta)^{\kappa/\alpha}},
  \end{align*}
  where the equality on the second line follows by comparing the
  integral with the previous display.  Since any compact subset of
  $\Omega$ is contained in a section $\{|\arg z|\le \beta_0\}$ for
  some $\beta_0<\kappa^{-1}\pi/2$ and since $d(\beta)>0$ is decreasing
  in $|\beta|$ in the section, it is then easy that $|g_t(z)| \le
  t^{-\ralph} \sup g_1/d(\beta_0)^{\kappa/\alpha}$, yielding the
  desired uniform boundedness. 
\end{proof}

We finally can prove the expression \eqref{e:residual2} for
$l_{x,-b,c}(t)$ and that it can analytically extended to
$\Coms\setminus (-\infty, b)$, which then finishes the proof of
Theorem \ref{t:exit-upper}.

\begin{proof}
  First, suppose $x\in (-b,0)$.  Then
  \begin{align*}
    \LT{l_{x,-b,c}}(q)
    =
    \frac{c^{\alpha-1} (b+x)^{\alpha-1}}{(b+c)^{\alpha-1}}
    \frac{
      E_{\alpha, \alpha}(c^\alpha q)
      E_{\alpha, \alpha}((b+x)^\alpha q)
    }{
      E_{\alpha, \alpha}((b+c)^\alpha q)
    }
  \end{align*}
  and \eqref{e:residual2} will follow once it is proved that $\FT 
  {l_{x,-b,c}}(\theta) = \LT {l_{x,-b,c}}(-\iunit\theta)$ is in
  $L^1(\dd\theta)$ and that
  \begin{align*}
    \nth{2\pi\iunit} \int^{\iunit\infty}_{-\iunit\infty}
    \frac{
      E_{\alpha, \alpha}(c^\alpha z)
      E_{\alpha, \alpha}((b+x)^\alpha z)
    }{
      E_{\alpha, \alpha}((b+c)^\alpha z)
    } e^{z t}\,\dd z
  \end{align*}
  is equal to the sum on the \rhs of \eqref{e:residual2}.  The
  argument is similar to that for Proposition \ref{p:exit-lower}.  Let
  \begin{align} \label{e:sv}
    s = c^\alpha/(b+c)^\alpha, \quad
    v = (b+x)^\alpha/(b+c)^\alpha.
  \end{align}
  Then by making change of variables $z' = (b+c)^\alpha z$ and $t' =
  t/(b+c)^\alpha$, and using the function $H_s(z)$ defined in
  \eqref{e:H_s(q)}, it boils down to showing that
  \begin{align} \label{e:l-ft}
    \intii |H_s(-\iunit\theta) E_{\alpha,\alpha}(-\iunit
    v\theta)|\,\dd\theta <\infty
  \end{align}
  and for any $t>0$,
  \begin{align} \label{e:l-contour}
    \int_{C_{R_n}} 
    H_s(z) E_{\alpha,\alpha}(v z) e^{t z}\,\dd z\to0, \quad n\toi,
  \end{align}
  where the contour $C_R$ and the numbers $R_n$ are defined in the
  proof of Proposition \ref{p:exit-lower}.

  By \eqref{e:asym-infty2}, with $\lambda = \cos(\alpha^{-1}\pi/2)>0$, 
  \begin{align*}
    |H_s(-\iunit\theta) E_{\alpha,\alpha}(-\iunit
    v\theta)| \sim 
    \alpha^{-1} |s v\theta|^{\ralph-1} e^{\lambda
      (s^\ralph + v^\ralph-1) |\theta|^\ralph}, \quad |\theta|\toi.
  \end{align*}
  Then \eqref{e:l-contour} easily follows by noting
  \begin{align} \label{e:negative}
    s^\ralph + v^\ralph - 1 = x/(b+c)<0.
  \end{align}
  Next, as in the proof of Proposition \ref{p:exit-lower}, divide
  $C_R$ into $C_{R,1}$ and $C_{R,2}$.  As $R\toi$, uniformly for $z\in
  C_{R,1}$, $H_s(z) = O(1) \exp\{(s^\ralph-1) z^\ralph\}$ and
  $E_{\alpha,\alpha}(v z) = O(1) (v z)^{\ralph-1} \exp\{v^\ralph
  z^\ralph\}$; see the derivation of \eqref{e:contour1}.  From
  \eqref{e:negative} again, there is $\eta = \eta(x)>0$, such that
  $\sup_{z\in C_{R,1}} |H_s(z) E_{\alpha, \alpha}(v z)| = O(e^{-\eta
    R})$.  Meanwhile, $|e^{t z}|\le 1$ for $z\in C_{R,1}$ and
  Length$(C_{R,1}) = O(R^\alpha)$.  Then
  \begin{align*}
    \int_{C_{R,1}} H_s(z) E_{\alpha,\alpha}(v z) e^{t z}\,\dd z
    = O(R^\alpha e^{-\eta R}) \to0,
    \quad R\toi.
  \end{align*}
  On the other hand, according to derivation of \eqref{e:contour2},
  for some $m_0>0$,
  \begin{align*}
    \sup_{z\in C_{R_n,2}} |H_s(z) E_{\alpha,\alpha}(v z)|
    =
    O(R^{1+\alpha}_n e^{m_0 R_n})\cdot O(R^{1-\alpha}_n e^{m_0 R})
    = O(R^2_n e^{2 m_0 R_n}).
  \end{align*}
  Meanwhile, according to the argument leading to \eqref{e:contour3},
  for some $b_0>0$, $|e^{t z}| \le e^{-b_0 R^\alpha t}$ for $t\in
  C_{R,2}$.  Then by $\alpha>1$ and $t>0$, 
  \begin{align*}
    \int_{C_{R_n,2}} H_s(z) E_{\alpha,\alpha}(v z) e^{t z}\,\dd z
    =
    O(R^{2+\alpha}_n e^{2m_0 R_n - b_0 R^\alpha_n t}),
    \quad n\toi.
  \end{align*}
  The desired convergence in \eqref{e:l-contour} then follows and
  hence \eqref{e:residual2} is proved in the case $x\in (-b,0)$.
  
  It only remains to show that the \rhs of \eqref {e:residual2} as a
  function of $x$ has an analytic extension to $\Coms\setminus
  (-\infty, -b]$ for given $t>0$.  Once this is done, since by Lemma
  \ref{e:l-extend}, $x\to l_{x,-b,c}(t)$ has an analytic extension to
  a domain containing $(-b,c)$ and since it was just proved that the
  two functions are equal on $(-b,0)$, then they must be equal on
  $(-b,c)$ and $x\to l_{x,-b,c}(t)$ can actually be extended to
  $\Coms\setminus (-\infty, -b]$, finishing the proof.
  
  Thus, let
  \begin{align*}
    w_\root(x) =
    \res\Grp{
      \frac{
        E_{\alpha,\alpha}(c^\alpha z) E_{\alpha,\alpha}((b+x)^\alpha
        z)
      }{
        E_{\alpha,\alpha}((b+c)^\alpha z)
      } e^{t z}, \frac\root{(b+c)^\alpha}
    }.
  \end{align*}
  It is easy to see that $w_\root$ has an analytic extension from
  $(-b,c)$ to $\Coms \setminus (-\infty, -b]$.  All $\root\in\Scr
  Z_{\alpha, \alpha}$ with large enough modulus are simple roots of
  $E_{\alpha, \alpha}(z)$ and have $|\arg\root|$ arbitrarily close but
  strictly greater than $\alpha\pi/2$.  For each such $\root$ and each
  $z\in \Coms\setminus (-\infty,0]$,
  \begin{align*}
    w_\root(z-b) =
    \frac{E_{\alpha,\alpha}(s\root)e^{t'\root}}{(b+c)^\alpha
      E'_{\alpha,\alpha}(\root)} \times 
    E_{\alpha, \alpha}\Grp{\frac{z^\alpha \root}{(b+c)^\alpha}},
  \end{align*}
  where $s$ is defined as in \eqref{e:sv} and $t' = t/(b+c)^\alpha$.
  From the last part of the proof for Proposition \ref{p:exit-lower},
  there is $c' = c'(t')>0$, such that
  \begin{align*}
    \Abs{\frac{E_{\alpha,\alpha}(s\root) e^{t'\root}}
      {(b+c)^\alpha E'_{\alpha,\alpha}(\root)}      
    } = O(1) e^{-c'|\root|}.
  \end{align*}
  On the other hand, by \eqref{e:asym-infty}, there is a constant
  $C>0$ such that
  \begin{align*}
    \Abs{
      E_{\alpha,\alpha}\Grp{\frac{z^\alpha\root}{(b+c)^\alpha}}
    }
    \le E_{\alpha,\alpha}\Grp{
      \frac{(|z|\vee 1)^\alpha |\root|}{(b+c)^\alpha}
    }= O(1)
    \exp\Cbr{C(|z|\vee 1) |\root|^\ralph}.
  \end{align*}
  Combining the two bounds then yields a bound on $|w_\root(z-b)|$.   
  Following the last part of the proof of Proposition \ref
  {p:exit-lower}, $\sum_{\root\in \Scr Z_{\alpha, \alpha}}|w_\root(z -
  b)|$ converges uniformly for $z$ in any compact subset of $\Coms
  \setminus(-\infty, 0]$.  Then $z\mapsto\sum_{\root\in \Scr
    Z_{\alpha, \alpha}} w_\root(z-b)$ is continuous on $\Coms
  \setminus(-\infty, 0]$.  Then by Fubini's  theorem followed by
  Morera's theorem, the \rhs of \eqref{e:residual2} as a function of
  $x$ has an analytic extension to $\Coms\setminus(-\infty, -b]$. 
\end{proof}

\subsection{Asymptotics}
We consider the asymptotics of $l_{x,-b,c}(t)$ as $t\dto0$ or
$\toi$.  First, we have
\begin{prop} \label{p:asym-l-0}
  Given $b>0$, $c>0$, and $x\in (-b,c)$, as $t\dto0$, $l_{x,-b,c}(t)
  \sim g_t(x)$.
\end{prop}
\begin{proof}
  It is clear that $l_{x,-b,c}(t) < g_t(x)$.  On the other hand,
  \begin{align*}
    g_t(x) - l_{x,-b,c}(t) 
    &\le
    \frac{\pr\{X_t\in \dd x,\, \tau_{-b}<t\}}{\dd x}
    +
    \frac{\pr\{X_t\in \dd x,\, T_c<t\}}{\dd x}
    \\
    &=
    \frac{\pr\{X_t\in \dd x,\, \tau_{-b}<t\}}{\dd x}
    +
    \frac{\pr\{X_t\in \dd x,\, \tau_{c}<t\}}{\dd x}
    \\
    &=
    \frac{\pr\{X_t\in \dd x,\, \tau_{-b}<t\}}{\dd x}
    +
    \frac{\pr\{X_t\in \dd x,\, \tau_{-(c-x)}<t\}}{\dd x},
  \end{align*}
  where the third line follows by considering $\sup\{s<t: X_s=c\}$ as
  well as time reversal.  Note that both $b$ and $c-x$ are greater
  than $(-x)_+$.  Then it suffices to show that for any $x\in\Reals$
  and $b>(-x)_+$, $j(t):=\pr\{X_t\in\dd x,\, \tau_{-b}<t\}/\dd x =
  o(g_t(x))$ as  $t\to 0$.

  By strong Markov property and $\sup g_s = O(s^{-1/\alpha})$, 
  \begin{align*}
    j(t) = \int^t_0 f_{-b}(t-s) g_s(x+b)\,\dd s
    = O(1) t^{1-1/\alpha} \sup_{s\le t} f_{-b}(s).
  \end{align*}
  Since $f_{-b}$ is unimodal (\cite{sato:99:cup}, p.~416), then for
  small $t>0$, $j(t) = O(1) f_{-b}(t)$.  On the other hand,
  \begin{align*}
    g_t(x) \asymp
    \begin{cases}
      t & \text{if } x>0,\\
      t^{-1/\alpha} & \text{if } x=0,\\
      t f_x(t)/(-x) & \text{if } x<0
    \end{cases}
    \quad\text{as } t\dto 0,
  \end{align*}
  where the last line is by Corollary VII.3 in \cite{bertoin:96:cup}.
  It is then clear that if $x\ge0$, then $j(t) = o(g_t(x))$.  On
  the other hand, if $x<0$, since $b>|x|$, then $f_{-b}(t) =
  o(f_{-x}(t))$, again yielding $j(t) = o(g_t(x))$.  
\end{proof}

Finally, by the same argument for the tail of $k_{-b,c}(t)$,
as $t\toi$, $l_{x,-b,c}(t)$ decreases exponentially fast with
\begin{align*}
  \limsup_{t\toi}
  \frac{\ln l_{x,-b,c}(t)}{t} = -\frac{\varrho}{(b+c)^\alpha},
\end{align*}
where $-\varrho < 0$ is the largest real root of $E_{\alpha,\alpha}$
and $n$ its multiplicity.  Again, the exact asymptotic of
$l_{x,-b,c}(t)$ depends on more detail of the roots along the line
$\Re(z) = -\varrho$.

\end{document}